\documentclass[preprint,11pt, english]{elsarticle}
\usepackage{amsmath}
\usepackage{latexsym, amssymb}
\usepackage{amsthm}
\usepackage{color}
\usepackage{txfonts}
\usepackage{calrsfs}

\usepackage{graphics}
\usepackage{rotating}
\usepackage{comment}

\input xy
\input xyidioms.tex
\usepackage{xy}
\xyoption{all} %

\newtheorem{thm}{Theorem}[section] 

\newtheorem{cor}[thm]{Corollary}
\newtheorem{lem}[thm]{Lemma}
\newtheorem{prop}[thm]{Proposition}

\theoremstyle{definition}
\newtheorem{rem}[thm]{Remark}

\newcommand\operA[2]{{\if!#2!\operatorname{#1}\else{\operatorname{#1}_{#2}^{\phantom{I}}}\fi}} 

\newcommand\Cref[1]{{Corollary~\ref{#1}}}%

\renewcommand{\phi}{\varphi}

\renewcommand{\phi}{\varphi}

\newcommand{\Trace}[1][]{\if!#1!\operatorname{Tr}\else{\operatorname{Tr}_{#1}^{\phantom{I}}}\fi} 

\long\def\forget#1\forgotten{{}} %

\def\({\left(}
\def\){\right)}

\newif\iffurther
\furtherfalse

\newif\ifXY 
\XYtrue     
\ifXY

\input xy
\input xyidioms.tex
\usepackage{xy}
\xyoption{all} %
\fi 

\usepackage{babel}

\begin{document}

\begin{frontmatter}

\title{Sums of two symbols in $K_2(F)/2K_2(F)$ in characteristic two}

\author[1]{Demba Barry}
\ead{barry.demba@gmail.com}

\author[3]{Adam Chapman}
\ead{adam1chapman@yahoo.com}

\author[4]{Ahmed Laghribi}
\ead{ahmed.laghribi@univ-artois.fr}

\address[1]{Facult\'e des Sciences et Techniques de Bamako, BP: E3206 Bamako, Mali}

\address[3]{School of Computer Science, Academic College of Tel-Aviv-Yaffo, Rabenu Yeruham St., Yaffo 86162, Israel}

\address[4]{Universit\'e d'Artois, Facult\'e des Sciences Jean Perrin, Laboratoire de math\'ematiques de Lens EA 2462, rue Jean Souvraz - SP18, 62307 Lens, France}

\begin{abstract}
In this paper, study sums $A=\{a,b\}_2+\{c,d\}_2$ of two symbols in $K_2(F)/2K_2(F)$ when $\operatorname{char}(F)=2$. We first prove a chain lemma that connects $A$ to $B=\{\alpha,\beta\}_2+\{\gamma,\delta\}_2$ by a finite sequence of small steps when $A \equiv B$. We use this lemma to prove that $\{a,b,c,d\}_2 \in K_4(F)/2K_4(F)$ is a well-defined invariant of $A$, and that this invariant is trivial if and only if $A$ is congruent to a single symbol in $K_2(F)/4K_2(F)$. We also bound the symbol length of $C$ in $K_2(F)/2^m K_2(F)$ from above when $C$ is the sum of up to four symbols in $K_2(F)/2^{m+1}K_2(F)$.
\end{abstract}

\date{\today}
\end{frontmatter}

\section{Introduction}

The mod 2 $K$-theory of fields $F$ of characteristic 2 gained significance the moment Kato showed in \cite{Kato:1982} that $K_n(F)/2K_n(F)$ is isomrophic to $I^n F/I^{n+1} F$ where $I F$ is the fundamental ideal of even-dimension forms in the Witt group $WF$ of symmetric bilinear forms over $F$.

In this paper we focus on sums $A=\{a,b\}_2+\{c,d\}_2$ of two symbols in $K_2(F)/2K_2(F)$.
 We first prove a chain lemma that connects $A$ to $B=\{\alpha,\beta\}_2+\{\gamma,\delta\}_2$ by a finite sequence of small steps when $A \equiv B$. We use this lemma to prove that $\{a,b,c,d\}_2 \in K_4(F)/2K_4(F)$ is a well-defined invariant of $A$, and that this invariant is trivial if and only if $A$ is congruent to a single symbol in $K_2(F)/4K_2(F)$. We also bound the symbol length of $C$ in $K_2(F)/2^m K_2(F)$ from above when $C$ is the sum of up to four symbols in $K_2(F)/2^{m+1}K_2(F)$.

All the results obtained in this paper have characteristic not 2 analogues,  see \cite{Sivatski:2012}, \cite{Sivatksi:2019},  \cite{Sivatski:2022} and \cite{Chapman:2023}. We refer to these analogous results, and when possible, we point out improvements in the existing literature that can be made by mimicking the technique we present here.

\section{Preliminaries}

Given a field $F$, its Milnor $K$-groups $K_n(F)$ are defined to be the groups of formal sums of symbols $\{\alpha_1,\alpha_2,\dots,\alpha_n\}$ with $\alpha_1,\alpha_2,\dots,\alpha_n \in F^\times$ subject to the following relations:
\begin{itemize}
\item If $\alpha_m=\beta_m$ for all $m \in \{1,\dots,n\}\setminus \{k\}$, then $\{\alpha_1,\dots,\alpha_n\}+\{\beta_1,\dots,\beta_n\}=\{\gamma_1,\dots,\gamma_n\}$ where $\gamma_m=\alpha_m$ for $m\neq k$ and $\gamma_k=\alpha_k \beta_k$.
\item If $\alpha_m=1-\alpha_k$ for some $m,k \in\{1,\dots,n\}$ with $m\neq k$, then $\{\alpha_1,\alpha_2,\dots,\alpha_n\}=0$.
\end{itemize}
A special attention is given to $K_n/2K_n(F)$. The class of  $\{\alpha_1,\alpha_2,\dots,\alpha_n\}$ in $K_n/2K_n(F)$ is denoted by $\{\alpha_1,\alpha_2,\dots,\alpha_n\}_2$. More generally, the class of $\{\alpha_1,\alpha_2,\dots,\alpha_n\}$ in $K_n/2^mK_n(F)$ is denoted by $\{\alpha_1,\alpha_2,\dots,\alpha_n\}_{2^m}$.

In \cite{Kato:1982}, it was proven that when $\operatorname{char}(F)=2$, $K_n(F)/2K_n(F)$ is isomorphic to $I^n F/I^{n+1} F$, by mapping $\{\alpha_1,\alpha_2,\dots,\alpha_n\}_2$ to $\langle \! \langle \alpha_1,\dots,\alpha_n\rangle \! \rangle$. The analogous result in $\operatorname{char}(F)\neq 2$ was proven in \cite{Voevodsky:2003}.

The ring of symmetric bilinear forms modulo metabolic forms is known as the Witt ring, and denoted $W F$. Its fundamental ideal consisting of even-dimensional forms is $I F$. Its $n$th power, $I^n F$, is known to be generated by scalar multiples of $n$-fold Pfister forms $\langle \! \langle \alpha_1,\dots,\alpha_n \rangle \! \rangle$, defined to be $\langle 1,-\alpha_1 \rangle \otimes \dots \otimes \langle 1,-\alpha_n \rangle$. For background, see \cite{EKM}.
\section{Chain Lemma}


Let $\tau \in K_2(F)/2K_2(F)$ be a sum of two symbols, and let $(V,E)$ be a labeled graph whose vertices are all pairs $(S_1,S_2)$ of symbols in $K_2(F)/2K_2(F)$ satisfying $S_1+S_2=\tau$. 

Two vertices $(S_1,S_2)$ and $(T_1,T_2)$ are connected by an edge if there exist $\alpha,\beta,\gamma,\delta \in F^\times$ and $x\in F$ such that: $S_1=\{\alpha,\beta\}_2$, $S_2=\{\gamma,\delta\}_2$, $T_1=\{\alpha,(x^2+\alpha\gamma)\beta\}_2$ and $T_2=\{\gamma,(x^2+\alpha\gamma)\delta\}_2$. If $x=0$ ({\it resp.} $x\neq 0$), then we say that $(S_1,S_2)$ and $(T_1,T_2)$ are connected by an edge of type I ({\it resp.} of type II).
Note that when $x\neq 0$, we can assume $x=1$.

\begin{lem}\
\begin{enumerate}
\item[(1)] If there exist $\alpha,\beta,\gamma,\delta \in F^\times$ and $x,y\in F$ not both zero such that $S_1=\{\alpha,\beta\}_2$, $S_2=\{\gamma,\delta\}_2$, $T_1=\{\alpha,(x^2+\alpha\gamma y^2)\beta\}_2$ and $T_2=\{\gamma,(x^2+\alpha\gamma y^2) \delta\}_2$, then $(S_1,S_2)$ and $(T_1,T_2)$ are either equal or connected by an edge of type I or II.
\item[(2)] If $(S_1,S_2)$ is connected to $(T_1,T_2)$ by an edge of type 2, then there exists a vertex $(R_1,R_2)$ connected to both $(S_1,S_2)$ and $(T_1,T_2)$ by edges of type I.
\end{enumerate}
\end{lem}

\begin{proof}\
\begin{enumerate}
\item[(1)] If $y=0$, then $(S_1,S_2)=(T_1,T_2)$. So suppose $y\neq 0$. If $x=0$, then $(S_1,S_2)$ is connected to $(T_1,T_2)$ by an edge of type I, and if $x\neq 0$, by an edge of type 2.
\item[(2)] Suppose that $(S_1,S_2)$ is connected to $(T_1,T_2)$ by an edge of type 2. Then, there exist $\alpha,\beta,\gamma,\delta \in F^\times$ such that $S_1=\{\alpha,\beta\}_2$, $S_2=\{\gamma,\delta\}_2$, $T_1=\{\alpha,\beta(1+\alpha\gamma)\}_2$ and $T_2=\{\gamma,\delta(1+\alpha\gamma)\}_2$.
Now, take $R_1=\{\alpha \gamma,\beta\}_2$ and $R_2=\{\gamma,\beta \delta\}_2$. Then, clearly $(S_1,S_2)$ is connected to $(R_1,R_2)$ by an edge of type I. Seeing that $R_1=\{\alpha \gamma,\beta(1+\alpha \gamma)\}_2$ clarifies why $(R_1,R_2)$ is also connected to $(T_1,T_2)$ by an edge of type I.
\end{enumerate} 
\end{proof}

\begin{thm}\label{chain}
Every two vertices are connected by a path consisting of at most three edges, with at most one of type II and the others of type I. If three edges are necessary, including one of type II, then the one of type II is in between the two edges of type I. If we restrict to edges of type I only, there is a path of length at most 4.
\label{t1}
\end{thm}

\begin{proof}
Suppose $\{\alpha,\beta\}_2+\{\gamma,\delta\}_2=\{a,b\}_2+\{c,d\}_2$.
Then the symmetric bilinear form $$\varphi=\langle \alpha,\beta,\alpha \beta \gamma,\alpha \beta \delta,\alpha \beta \gamma \delta, a,b,abc,abd,abcd \rangle$$ is a 10-dimensional form in $I^3F$. By \cite[Proposition 6.5]{ChapmanLaghribiMukhija:2025}, $\varphi$ is isotropic.
Therefore, $\alpha x^2+\beta y^2+\alpha \beta (z^2\gamma+w^2\delta +v^2\gamma \delta)=a q^2+b r^2+ab (c s^2+d t^2+cdu^2)$ for some $q,r,s,t,u,v,w,x,y,z \in F$, not all zero.

Now, $\{\alpha,\beta\}_2+\{\gamma,\delta\}_2=\{\alpha x^2+\beta y^2,\alpha \beta\}_2+\{z^2\gamma+w^2\delta +v^2\gamma \delta,\gamma\delta(z^2 \gamma+w^2 \delta)\}_2\\=
\{\alpha x^2+\beta y^2,\alpha \beta(z^2\gamma+w^2\delta +v^2\gamma \delta)\}_2+\{z^2\gamma+w^2\delta +v^2\gamma \delta,\gamma\delta(z^2 \gamma+w^2 \delta)(\alpha x^2+\beta y^2)\}_2
\\=\{\alpha x^2+\beta y^2+\alpha \beta(z^2\gamma+w^2\delta +v^2\gamma \delta),\alpha \beta(z^2\gamma+w^2\delta +v^2\gamma \delta)(\alpha x^2+\beta y^2)\}_2+\{z^2\gamma+w^2\delta +v^2\gamma \delta,\gamma\delta(z^2 \gamma+w^2 \delta)(\alpha x^2+\beta y^2)\}_2$.

Similarly, $\{a,b\}_2+\{c,d\}_2=\{a q^2+b r^2,ab\}_2+\{c s^2+dt^2 +cdu^2,cd(c s^2+d t^2)\}_2\\=
\{a q^2+b r^2,ab(c s^2+dt^2 +cdu^2)\}_2+\{c s^2+dt^2 +cdu^2,cd(c s^2+dt^2 +cdu^2)(a q^2+b r^2)\}_2
\\=\{a q^2+b r^2+ab(c s^2+dt^2 +cdu^2),ab(c s^2+dt^2 +cdu^2)(a q^2+b r^2)\}_2+\{c s^2+dt^2 +cdu^2,cd(c s^2+dt^2 +cdu^2)(a q^2+b r^2)\}_2$.

Therefore, $(\{\alpha,\beta\}_2,\{\gamma,\delta\}_2)$ is connected to $(\{A,B\}_2,\{C,D\}_2)$ by an edge of type I, and $(\{a,b\}_2,\{c,d\}_2)$ is connected to $(\{A,L\}_2,\{M,N\}_2)$ by an edge of type I.
It is enough to show that $(\{A,B\}_2,\{C,D\}_2)$ is connected to $(\{A,L\}_2,\{M,N\}_2)$ by an edge of either type.
Since $\{A,B\}_2+\{C,D\}_2=\{A,L\}_2+\{M,N\}_2$, we have $\{A,BL\}_2=\{C,D\}_2+\{M,N\}_2$. Therefore, the symmetric bilinear form $$\langle C,D,CD,M,N,MN \rangle$$ is isotropic by \cite[Section 5, Lemma 1]{ChapmanLaghribi:2025}.
Consequently, $C i^2+Dj^2+CDk^2=M\ell^2+Nm^2+MNn^2$, which means $\{C,D\}_2=\{Ci^2+Dj^2+CDk^2,CD(Ci^2+Dj^2)\}_2$ and $\{M,N\}_2=\{M\ell^2+Nm^2+MNn^2,MN(M\ell^2+Nm^2)\}_2$.
It means that we could assume $C=M$ from the start.
Now, $\{A,B\}_2+\{C,D\}_2=\{A,L\}_2+\{C,N\}_2$ implies $\{A,BL\}_2=\{C,DN\}_2$. Hence, there exists $P$ for which $\{A,BL\}_2=\{A,P\}_2=\{C,P\}_2=\{C,DN\}_2$. 
Since $\{AC,P\}_2=0$, $P=\mu^2+AC\nu^2$ for some $\mu,\nu \in F$. From $\{A,BLP\}_2=0$ we conclude that by a suitable choice of $B$, we can assume $L=BP$. Similarly, from $\{C,DN\}_2=0$, by a suitable choice of $D$, we can assume $N=DP$, and that completes the proof.
\end{proof}

\begin{rem}
Note that what we call here an edge of type I is analogous what is called in \cite{Sivatski:2012} ``strongly simply-equivalent". In that paper, the author concluded that in characteristic not 2, every two vertices in this graph are connected by a path of length at most three of edges of either type and then concluded that a path of length at most 6 of edges of type I must exist (\cite[Propositions 1 and 2]{Sivatski:2012}). However, following the same computations as we describe here, his Proposition 2 can be easily improved to a path of length at most 4 instead of 6.
\end{rem}

The following corollary is the characteristic 2 analogue of the main result of \cite{Sivatksi:2019}. It is also a special case of \cite[Appendix]{Kahn:2000}, but it is nice to have an elementary explanation.
\begin{cor}
The map from sums of two symbols in $K_2(F)/2K_2(F)$ to $K_4(F)/2K_4(F)$ given by $\{a,b\}_2+\{c,d\}_2 \mapsto \{a,b,c,d\}_2$ is well-defined.
\end{cor}

\begin{proof}
It readily follows from Theorem \ref{t1}. It is enough to stress that an edge of type I between two vertices does not change the image of this sum under this map. Indeed, we have: $\{a,bc\}_2+\{c,ad\}_2\mapsto \{a,bc,c,ad\}_2=\{a,b,c,d\}_2$.
\end{proof}

\begin{lem}
If the Albert form $\left<a, b, ab, c, d, cd\right>$ is anisotropic, then the distance between $(\{a, b\}_2, \{c, d\}_2)$ and $(\{c, d\}_2, \{a, b\}_2)$ is bigger than $1$.
\label{l1}
\end{lem}
\begin{proof}
If $(\{a, b\}_2, \{c, d\}_2)$ is at distance $1$ from $(\{c, d\}_2, \{a, b\}_2)$, then $\{a, b\}_2$ and $\{c, d\}_2$ share a common slot, and thus $\left<a, b, ab, c, d, cd\right>$ is isotropic.
\end{proof}

\begin{prop}
Suppose that the Albert form attached to $\tau$ has norm degree $16$. Then, the diameter of the graph $(V,E)$ is exactly 3.
\end{prop}

\begin{proof} Take a vertice $(\{a,b\}_2,\{c,d\}_2)$. By our hypothesis on norm degree, we have $[F^2(a,b,c,d):F^2]=2^4$. It is enough to explain why $(\{a,b\}_2,\{c,d\}_2)$ is at distance no less than 3 from $(\{c,d\}_2,\{a,b\}_2)$. This distance is at least $2$ by Lemma \ref{l1}. Now, suppose that $(\{c,d\}_2,\{a,b\}_2)$ is at distance $2$ from $(\{a,b\}_2,\{c,d\}_2)$. Then, there are suitable scalars $a',b',c',d'$ and $x\in \{0, 1\}$, such that $(\{a,b\}_2,\{c,d\}_2)=(\{a',b'\}_2,\{c',d'\}_2)$ and $(\{c, d\}_2,\{a, b\}_2)$ is at distance 1 from$$(\{a',b'(x^2+a'c')\}_2,\{c',d'(x^2+a'c')\}_2).$$Then, there is at least one nonzero element $\alpha:=a'q^2+b'(x^2+a'c')r^2+a'b'(x^2+a'c')t^2$ which is common slot of $\{c, d\}_2$ and $\{a',b'(x^2+a'c')\}_2$. Hence, $F^2(\alpha)\subset F^2(c, d)\cap F^2(a', b', c')$. Moreover, using the uniqueness of the pure part of a bilinear Pfister form, the equalities $\{a,b\}_2=\{a',b'\}_2$ and $\{c,d\}_2=\{c',d'\}_2$ imply $F^2(a, b)=F^2(a', b')$ and $F^2(c, d)=F^2(c',d')$. Hence, $F^2(a, b, c, d)=F^2(a', b', c', d')$. Note that $\alpha\not\in F^2(c')$, otherwise we would get $c'\in F^2(a', b')$. Hence, $F^2(c', d')=F^2(\alpha, c')$, and thus $[F^2(a, b, c, d):F^2]<16$, a contradiction. This implies that  $(\{a,b\}_2,\{c,d\}_2)$ cannot be of distance 2 from $(\{c,d\}_2,\{a,b\}_2)$.
\end{proof}

\section{Single symbols in $K_2(F)/4K_2(F)$ of exponent 2}

In this section we provide a characteristic 2 analogue of \cite{Sivatski:2022}, tackling the question of when a sum of two symbols in $K_2(F)/2K_2(F)$ is a single symbol in $K_2(F)/4K_2(F)$.
Recall here that there is a short exact sequence (see \cite[Remark 2.32]{AravireJacobORyan:2018})
\begin{equation}
0 \rightarrow K_2(F)/2^m K_2(F) \rightarrow K_2(F)/2^nK_2(F) \rightarrow K_2(F)/2^{n-m} K_2(F) \rightarrow 0
\label{e1}
\end{equation}
where the embedding map is $\{a,b\}_{2^m} \mapsto \{a,b^{2^{n-m}}\}_{2^n}$, and the following map is $\{a,b\}_{2^n} \mapsto \{a,b\}_{2^{n-m}}$.
We consider the first map as an embedding of $K_2(F)/2^mK_2(F)$ inside $K_2(F)/2^nK_2(F)$, i.e., $\{a,b^{2^{n-m}}\}_2^n=\{a,b\}_{2^m}$.

\begin{lem}
If $\{a,b\}_{2^{n+1}} \in K_2(F)/2^{n+1}K_2(F)$ is of exponent dividing $2^n$, then its symbol length in $K_2(F)/2^n K_2(F)$ is at most 2.
\end{lem}

\begin{proof}
Let $\{a,b\}_{2^{n+1}} \in K_2(F)/2^{n+1}K_2(F)$ be of exponent $2^n$. Using the sequence (\ref{e1}), it follows $\{a,b\}_{2}=0$, which means $b=x^2+ay^2$ for $x, y\in F^{\times}$. Hence,
\begin{eqnarray*}
\{a,b\}_{2^{n+1}} &= &\{a,x^2+ay^2\}_{2^{n+1}}\\&=&\{ay^2,x^2+ay^2\}_{2^{n+1}}+\{y^2,x^2+ay^2\}_{2^{n+1}}\cr &=&\{x^2,(x^2+ay^2)ay^2\}_{2^{n+1}}+\{y^2,x^2+ay^2\}_{2^{n+1}}\cr 
&=& \{x,(x^2+ay^2)a\}_{2^n}+\{y,x^2+ay^2\}_{2^n}.
\end{eqnarray*}
\end{proof}

The case of $n=1$ in this lemma serves as a partial analogue of Albert's theorem on degree 4 exponent 2 central simple algebras (see \cite[Page 174, Theorem 2]{Albert:1968}).

\begin{thm}
A sum of two symbols $\tau=\{a,b\}_2+\{c,d\}_2\in K_2(F)/2K_2(F)$ is equal to a single symbol in $K_2(F)/4K_2(F)$ if and only if $\{a,b,c,d\}$ is trivial in $K_4(F)/2K_4(F)$.
\end{thm}

\begin{proof}
Suppose that $\tau$ is a symbol in $K_4(F)/2K_4(F)$. This means that $\{a, b^2\}+\{c, d^2\}=\{\alpha,\beta\}_4$ for suitable $\alpha, \beta \in F^{\times}$. Then, we get $\{\alpha,\beta\}_2=0$ by the sequence (\ref{e1}). The previous remark yields$$\tau=\{x,(x^2+\alpha y^2)\alpha\}_2+\{y,x^2+\alpha y^2\}_2,$$for suitable $x, y\in F^{\times}$. This means $\{a,b,c,d\}=\{x,(x^2+\alpha y^2)\alpha,y,x^2+\alpha y^2\}$ for some $x,y \in F$. Since $[F^2(x,(x^2+\alpha y^2)\alpha,y,x^2+\alpha y^2):F^2]= [F^2(\alpha,x,y):F^2]\leq 8$, the symbol $\{a,b,c,d\}$ must be trivial in $K_4(F)/2K_4(F)$.

In the opposite direction, suppose $\{a,b,c,d\}$ is trivial in $K_4(F)/2K_4(F)$. Then $\langle \! \langle a,b,c,d \rangle \! \rangle$ is isotropic. If $\tau$ is equal to one symbol in $K_2(F)/2K_2(F)$ then clearly it is equal to one symbol in $K_2(F)/4K_2(F)$, so assume it is not. In particular, the Albert form $\varphi=\langle a,b,ab,c,d,cd\rangle$ is anisotropic. Moreover, the isotropy of $\langle \! \langle a,b,c,d \rangle \! \rangle$ implies that its quasi-Pfister neighbor $$\langle 1,ac,ad,acd,a,ab,c,d,cd\rangle$$is isotropic.
Therefore, $q^2+a(cr^2+ds^2+cdt^2)+au^2+abv^2+cx^2+dy^2+cdz^2=0$ for some $q,r,s,t,u,v,x,y,z\in F$ not all zero.
Since $\varphi$ is anisotropic, we have $\lambda:=q^2+a(cr^2+ds^2+cdt^2)$ is not zero. Moreover, the scalars $cr^2+ds^2+cdt^2$ and $cx^2+dy^2+cdz^2$ are not simultaneously zero. Take $c'$ a nonzero scalar among them. Therefore, we can choose $d'$ such that $\{c,d\}_2=\{c',d'\}_2$ and $\lambda+au^2+abv^2+c'\ell^2+d'm^2+c'd'n^2=0$ for some $\ell,m,n \in F$.
Dividing by $\lambda$, we obtain that the form $\psi=\langle 1,\lambda a,\lambda ab,\lambda c',\lambda d',\lambda c'd' \rangle$ is isotropic.

Since $\langle 1,ac'\rangle$ represents $\lambda$, we have $\langle 1,ac'\rangle = \lambda \langle 1,ac' \rangle$, which means $\langle \lambda a,\lambda c'\rangle=\langle a,c'\rangle$.
Therefore,
$\psi=\langle 1,a,\lambda ab,c',\lambda d',\lambda c'd' \rangle$.
Since $\psi$ is isotropic, there exists an element $\gamma$ that is represented both by $\langle 1,a,\lambda ab \rangle$ and by $\langle c',\lambda d',\lambda c'd' \rangle$. 
Therefore, there exist $\delta,\beta \in F^\times$ such that $\{c',\lambda d'\}_2=\{\gamma,\delta\}_2$ and either $\{a,\lambda b\}_2=\{1+\gamma,\beta\}_2$ or $\{\gamma,\beta]_2$.
The second option is impossible, because that would mean that $\tau=\{a,b\}_2+\{c,d\}_2=\{a,b\}_2+\{c',d'\}_2+\{a,\lambda b\}_2+\{c',\lambda d'\}_2=\{\gamma,\delta \beta\}_2$ is equal to one symbol in $K_2(F)/2K_2(F)$, and we assumed it is not. 
Hence, $\{\gamma,\delta\}_2=\{\gamma,\delta^2\}_4=\{\gamma,\delta^2(\gamma+1)\}_4$, whereas $\{1+\gamma,\beta\}_2=\{(1+\gamma)\delta^2,\beta\}_2=\{(1+\gamma)\delta^2,\beta^2\}_4$. 
Consequently, $\tau=\{\gamma,\delta^2(\gamma+1)\}_4+\{(1+\gamma)\delta^2,\beta^2\}_4=\{\gamma \beta^2,\delta^2(\gamma+1)\}_4$.

\end{proof}

\begin{rem}
To complete the picture, one may want to know what happens for biquaternion algebras in characteristic 2. It is well-known that they are cyclic of degree 4, but here is an easy explanation: given a biquaternion algebra $[a,b)_2 +[c,d)_2$,
$[a,b)_2=[a+x^2+x,b)_2$ and $[c,d)_2=[c\frac{d+x^2}{d},d+x^2)_2$, so by taking $x=a+d$, we get $a+x^2+x=d+x^2$. Hence, we could assume from the start that $a=d$. 
Now, $[a,b)_2=[(a,0),b^2)_4=[(a,0),ab^2)_4$ and $[c,a)_2=[c,ab^2)_2=[(0,c),ab^2)_4$, and so $[a,b)_2+[c,d)_2=[(a,0),ab^2)_4+[(0,c),ab^2)_4=[(a,c),ab^2)_4$.
For the notation and further background, see \cite{MammoneMerkurjev:1991} and \cite{Chapman:2026}.
\end{rem}

\section{Symbol Length}

The symbol length of an element $\omega \in K_n(F)/2^m K_n F$ is the minimal number of symbols in $K_n(F)/2^m K_n F$ required to express it.
In this section we study what bounds on the symbol length of $\omega$ if we know its symbol length in $K_n(F)/2^{m+1} K_n(F)$.

\begin{thm}\label{bounds}
Consider a class $\omega \in K_2(F)/2^m K_2(F)$.
\begin{itemize}
\item[(a)] If its symbol length in $K_2(F)/2^{m+1}K_2(F)$ is 2, then  its symbol length in $K_2(F)/2^{m} K_2(F)$ is at most 5.
\item[(b)] If its symbol length in $K_2(F)/2^{m+1}K_2(F)$ is 3, then  its symbol length in $K_2(F)/2^{m} K_2(F)$ is at most 15.
\item[(c)] If its symbol length in $K_2(F)/2^{m+1}K_2(F)$ is 4, then  its symbol length in $K_2(F)/2^{m} K_2(F)$ is at most 40.
\end{itemize}
\end{thm}

\begin{proof}
If $\omega=\{\alpha,\beta\}_{2^{m+1}}+\{\gamma,\delta\}_{2^{m+1}}$, then $\{\alpha,\beta\}_2=\{\gamma,\delta\}_2$. Since $\{\alpha,\beta\}_{2^{m+1}}\neq 0$ because $\omega$ is of length $2$, we obtain $\left<\alpha, \beta, \alpha\beta\right>\cong \left<\gamma, \delta, \gamma\delta\right>$ (the uniqueness of the pure part of bilinear Pfister forms). Hence, $\gamma=\alpha x^2+\beta y^2+\alpha \beta z^2$ for some $x,y,z\in F$, and we have
$$\{\alpha,\beta\}_2=\{\alpha x^2,\beta(y^2+\alpha z^2)\}_2=\{\gamma,\alpha\beta(y^2+\alpha z^2)\}_2=\{\gamma,\delta\}_2.$$
Now, 
\begin{equation*}
\omega=\underbrace{\{\alpha,\beta\}_{2^{m+1}}-\{\alpha x^2,\beta\}_{2^{m+1}}}_{A}+ \underbrace{\{\alpha x^2,\beta\}_{2^{m+1}}-\{\alpha x^2,\beta(y^2+\alpha z^2)\}_{2^{m+1}}}_{B}+\underbrace{\{\gamma,\alpha\beta(y^2+\alpha z^2)\}_{2^{m+1}}+\{\gamma,\delta\}_{2^{m+1}}}_{C}.
\end{equation*}
Clearly, $A=\{x^{-2},\beta\}_{2^{m+1}}$ which is the same as $\{x^{-1},\beta\}_{2^m}$. We have $B=-\{\alpha x^2,y^2+\alpha z^2\}_{2^{m+1}}$ and $C=\{\gamma,\delta \alpha\beta(y^2+\alpha z^2)\}_{2^{m+1}}$ which are both single symbols in $K_1(F)/2^{m+1}K_1(F)$. The exponents of $B$ and $C$ divide $2^m$, because their $2^m$th powers are $\{\alpha x^2,\beta(y^2+\alpha z^2)\}_2-\{\gamma,\alpha\beta(y^2+\alpha z^2)\}_2=0$ and $\{\gamma,\alpha\beta(y^2+\alpha z^2)\}_2+\{\gamma,\delta\}_2=0$, respectively. Therefore, each is equal to a sum of up to two symbols in $K_2(F)/2^m K_2(F)$. This completes part $(a)$.

Now, if $\omega=\{\alpha,\beta\}_{2^{m+1}}+\{\gamma,\delta\}_{2^{m+1}}+\{a,b\}_{2^{m+1}}$, then $\{\alpha,\beta\}_{2}+\{\gamma,\delta\}_{2}+\{a,b\}_{2}=0$. This implies that there exist $c,d,e \in F^\times$, such that $\{\alpha,\beta\}_2 =\{c,d\}_2$, $\{\gamma,\delta\}_2=\{c,e\}_2$ and $\{a,b\}_2=\{c,ed\}_2$.
Hence, $\omega$ can be written as 
\begin{eqnarray*}
\{\alpha,\beta\}_{2^{{m}+1}}-\{c,d\}_{2^{{m}+1}}\\
+\{\gamma,\delta\}_{2^{m+1}}-\{c,e\}_{2^{{m}+1}}\\
+\{a,b\}_{2^{{m}+1}}+\{c,de\}_{2^{{m}+1}}.
\end{eqnarray*}
Each term is a sum of two symbols in $K_2 F/2^{{m}+1} K_2 F$ of exponent dividing $2^{{m}}$, and thus is the sum of up to 5 symbols in $K_{{2}}F/2^{{m}} K_{{2}} F$ by part $(a)$. Therefore, $\omega$ is the sum of up to 15 symbols in $K_{{2}} F/2^{{m}} K_{{2}} F$. This completes part $(b)$.

If $\omega=\{\alpha,\beta\}_{2^{n+1}}+\{\gamma,\delta\}_{2^{n+1}}+\{a,b\}_{2^{n+1}}+\{c,d\}_{2^{n+1}}$, then $\{\alpha,\beta\}_2+\{\gamma,\delta\}_2=\{a,b\}_2+\{c,d\}_2$. By Theorem \ref{chain}, there exist $q,r,s,t,w,x,y,z,e,f,g,h \in F$ such that $\{\alpha,\beta\}_2=\{q,r\}_2$, $\{\gamma,\delta\}_2=\{s,t\}_2$, $\{a,b\}_2=\{w,x\}_2$, $\{c,d\}_2=\{y,z\}_2$, $\{q,rs\}_2=\{e,f\}_2$, $\{s,qt\}_2=\{g,h\}_2$, and either $(i)$ $\{e,fg\}_2=\{w,x\}_2$ and $\{g,eh\}_2=\{y,z\}_2$ or  $(ii)$ $\{e,(1+eg)f\}_2=\{w,x\}_2$ and $\{g,(1+eg)h\}=\{y,z\}_2$. 
Suppose $(i)$. Then $\omega$ can be written as the sum of
\begin{eqnarray*}
\{\alpha,\beta\}_{2^{m+1}}-\{q,r\}_{2^{m+1}}\\
\{\gamma,\delta\_{2^{m+1}}-\{s,t\}_{2^{m+1}}\\
\{q,rs\}_{2^{m+1}}-\{e,f\}_{2^{m+1}}\\
\{s,qt\}_{2^{m+1}}-\{g,h\}_{2^{m+1}}\\
\{e,fg\}_{2^{m+1}}-\{w,x\}_{2^{m+1}}\\
\{g,eh\}_{2^{m+1}}-\{y,z\}_{2^{m+1}}\\
\{w,x\}_{2^{m+1}}+\{a,b\}_{2^{m+1}}\\
\{y,z\}_{2^{m+1}}+\{c,d\}_{2^{m+1}}.
\end{eqnarray*}
Each line is the sum of up to 5 symbols in $K_2 F/2^{m} K_2 F$ by part $(a)$. Therefore, $\omega$ is the sum of up to 40 symbols in $K_2 F/2^{m} K_2 F$. The case of $(ii)$ is similar.
\end{proof}

\begin{rem}
Note that in the analogous statement to part $(c)$ of Theorem \ref{bounds} in \cite{Chapman:2023}, i.e., Theorem 3.5 of that paper, the bound can be reduced from 46 to 40 by the refinement of Sivatski's chain lemma as we presented here and can be easily adapted to the case of characteristic not 2 without any significant change.
\end{rem}

\bibliographystyle{abbrv}
\bibliography{bibfile}
\end{document}